\documentclass{article}

\usepackage{amssymb,amsmath,amsthm}
\usepackage{MnSymbol}
\usepackage{color}
\usepackage{enumerate}
\usepackage{fullpage}
\usepackage{url}
\usepackage[unicode,pdfusetitle]{hyperref}
\usepackage[mathscr]{euscript}
\usepackage{graphicx}
\usepackage{mathtools}
\usepackage[T1]{fontenc}
\usepackage{authblk}
\usepackage{comment}

\newtheorem{theorem}{Theorem}
\newtheorem*{theorem*}{Theorem}
\newtheorem{fact}[theorem]{Fact}
\newtheorem*{fact*}{Fact}

\newtheorem*{claim*}{Claim}
\newtheorem{proposition}[theorem]{Proposition}
\newtheorem*{proposition*}{Proposition}
\newtheorem{lemma}[theorem]{Lemma}
\newtheorem*{lemma*}{Lemma}

\newtheorem*{question*}{Question}

\newtheorem*{conjecture*}{Conjecture}
\newtheorem{corollary}[theorem]{Corollary}
\newtheorem*{corollary*}{Corollary}

\theoremstyle{definition}
\newtheorem{definition}[theorem]{Definition}
\newtheorem*{definition*}{Definition}
\theoremstyle{remark}
\newtheorem{remark}[theorem]{Remark}
\newtheorem*{remark*}{Remark}
\newtheorem{example}[theorem]{Example}
\newtheorem*{example*}{Example}

\numberwithin{theorem}{section}
\numberwithin{claim}{section}
\numberwithin{equation}{section}

\newcommand{\bbZ}{\mathbb{Z}}

\newcommand{\cL} {\mathcal{L}}
\newcommand{\cF} {\mathcal{F}}
\newcommand{\bbR}{\mathbb{R}}

\title{A note on uniform definability of types over finite sets in partial orders of finite width}
\author{Timo Krisam\thanks{Supported by a Minerva Fellowship of the Minerva Stiftung Gesellschaft fuer die
Forschung mbH} \ 
and\
Ori Segel}
\begin{document}
\maketitle

\begin{abstract}
In \textit{VC density in some theories without the independence property} the authors asked whether any partial order of finite width has the VC1 property (i.e. every formula in one variable has UDTFS in one parameter). We give a negative answer and some related remarks.
\end{abstract}

\section{Introduction and Preliminaries}

We give a brief introduction to UDTFS and the VCd property. 
Our presentation follows \cite{aschenbrennervc}.\\
For this section, we fix a first order language $\cL$, an $\cL$-structure $M$.
$\Delta(x;y)$ will denote a set of partitioned $\cL$-formulas, where both $x$ and $y$ are tuples of finite length.
$S^\Delta(B)$ denotes the set of $\Delta$-types over a set $B \subseteq M^{|y|}$.

\begin{definition}
    Let $q \in S^\Delta(B)$ for some $\Delta(x;y)$. 
    A family of $\cL (M)$-formulas $\cF = (\phi_\#(y))_{\phi \in \Delta}$ \textit{defines} $q$ if for all $\phi \in \Delta$ and $b \in B$ we have that \[
     \phi(x;b) \in q\Leftrightarrow M \vDash \phi_\#(b).
    \]
 \end{definition}

 \begin{definition}[UDTFS]
     $\Delta(x;y)$ has \textit{uniform definability of types over finite sets (UDTFS) (in d parameters)} in $M$, if there are finitely many families $\cF_i = (\phi_i(y;y_1,...,y_d))_{\phi \in \Delta}$ with $|y| = |y_j|$, $1 \leq j \leq d$ and $i \leq n$ for some $n < \omega$, such that for any finite $B \subseteq M^{|y|}$ and $q \in S^\Delta(B)$ there are $b_1,...,b_d \in B$ and $i \leq n$ such that $\cF_i(y;b_1,...,b_d)$ defines $q$.\\
     If $\Delta = \{\phi\}$, we also say that $\phi$ has UDTFS (in $d$ parameters).
 \end{definition}

 \begin{definition}[VCd]
     Let $d < \omega$. We say that $M$ has the \textit{VCd property}, if every $\Delta(x;y)$ with $|x| = 1$ has UDTFS in $d$ parameters.\\
     If $T$ is a theory, we say that $T$ has the VCd property, if every $M \vDash T$ does.
 \end{definition}

The authors go on to show that weakly quasi o-minimal theories have the VC1 property.
To do that, they use the following tools:

\begin{fact}[\cite{aschenbrennervc}, Lemma 5.5]\label{bool}
    Let $\Delta(x;y)$ and $\Delta'(x;y)$ be finite. Suppose that every formula in $\Delta$ is equivalent to a boolean combination of formulas in $\Delta'$.\\
    Then, if $\Delta'$ has UDTFS in $d$ parameters, so does $\Delta$.
\end{fact}

\begin{definition}
    Let $X$ be a set and $\cF$ a collection of subsets of $X$.
    Suppose there is $d>0$ such that any non-empty intersection $\bigcap\limits_{i = 1}^n F_i$ for $F_i \in \cF$ and $n > d$ is equal to an intersection of $d$ of the $F_i$.
    Then we define the \textit{breadth} of $\cF$ to be the smallest integer $d$ such that $\cF$ has this property.
\end{definition}

\begin{fact}[\cite{aschenbrennervc}, Lemma 5.2]\label{breadth}
    Let $\Delta(x;y)$ be finite and suppose that $\{ \phi(M^{|x|}; b) \mid b \in M^{|y|} \}$ has breadth $d$. Then $\Delta$ has UDTFS in $d$ parameters.
\end{fact}

For the remainder of this note, we are interested in partial orders of finite width.

\begin{definition}
    A partial order $(M,<)$ has \textit{width} $n$ if every antichain in $M$ contains at most $n$ elements.
\end{definition}

The authors show that a partial order of finite width is interpretable in a weakly quasi o-minimal theory.
While this yields many interesting properties that we did not introduce here, VCd (in particular VC1) is not preserved under interpretation in general. Because of this, the authors asked the following question:

\begin{question*}
    Does every partial order of finite width have the VC1 property?
\end{question*}

\section{A counterexample of width 3}

\begin{example}
    Let $(M, <)$ be the structure with universe $\bbR \times \{0,1,2\}$ and
    \[M \vDash (x,i) < (y,j) \text{ if and only if } 
    \begin{cases}
    x <_{\bbR} y & i = j \\
    x + \frac{1}{2} <_{\bbR} y & i \neq j
    \end{cases}\]
    Clearly, $M$ is a partial order of width 3.
    \begin{claim*}
        $M$ does not have VC1. 
    \end{claim*}
    \begin{proof}
        Let $\phi(x;y) \equiv y < x$. Suppose there are $n < \omega$ and $(\psi_i(y; y'))_{i < n}$ witnessing VC1 for $\phi$.\\
        Consider the set $B = (\frac{\bbZ}{4n} \cap [0,1]) \times \{1,2\} \cup \{(\frac{1}{2},0)\}$.
        Let $a \in (\frac{1}{2}, 1) \times \{0\}$ and let $b \in B$ such that $\psi_i(y;b)$ defines 
       $p = tp_\phi(a/B)$ for some $i < n$. \\
       Then $\pi_2(b) = 0$. 
       Indeed, suppose we have $\pi_2(b) = i \neq 0$.\\
       Consider the automorphism 
       \[\sigma((x,j)) = 
       \begin{cases}
           (x,j) & i = j\\
           (x,0) & i + j = 3\\
           (x, 3-j) & j = 0
       \end{cases},\] 
       i.e the automorphism exchanging the other copies of $\bbR$.\\
       Then $p \vdash x > (\frac{1}{2}, 0)$ and $\sigma(p) \vdash \neg x > (\frac{1}{2}, 0)$ by definition of $p$ (note that $(\frac{1}{2}, 0) \in \sigma(B)$ by construction).
       But $\sigma(b) = b$, contradicting the assumption that $b$ defines $p$.

       As there is only one $b \in B$ with $\pi_2(b) = 0$, the $\psi_i$'s can define at most $n$ distinct types over $B$.
       But by choice of $B$, there are at least $n+1$ distinct $\phi$-types over $B$, contradicting our assumption.
    \end{proof}
\end{example}

\begin{remark}
    Clearly, we can produce orders of width $n$ without VC1 for $n \geq 3$ by using an analogue of the above construction on $\bbR \times \{0,...,n-1\}$.
\end{remark}

\begin{question*} This example raises multiple followup questions:
    \begin{itemize}
        \item Does every partial order of width 2 have the VC1 property?
        \item Let $n < \omega$. Is there some $d < \omega$ such that any partial order of width $n$ has VCd? 
        \item If not, is there an order of finite width without any of the VCd properties?
    \end{itemize}
\end{question*}
\begin{remark}
    Note that by Remark 3.7, even if the answer to the second question is positive, d must depend on n.
\end{remark}

Assuming quantifier elimination, we can answer the second question positively.

\begin{remark}
    Suppose $(M,<)$ has width $n$ and $\Delta(x;y)$ is a set of quantifier-free formulas with $|x| = 1$. 
    Then $\Delta$ has UDTFS in $n$ parameters.
\end{remark}
\begin{proof}
    By \ref{bool} (as $\Delta$ is finite), it is enough to show that the set $\{\psi(x;z) \equiv x = z, \phi(x;y) \equiv x < z\}$ has UDTFS in $n$ parameters.\\
    To do this, we show that the set $S = \{\phi(M,b) \mid b \in M\}$ has breadth $n$. 
    This is enough by \ref{breadth}, noting that $\psi(M,b) = \{b\}$ for any $b \in M$.\\
    Let $b_0,...,b_n \in M$. By assumption, there are $i \neq j$ such that $b_i > b_j$.
    But then $\phi(M;b_i) \cap \phi(M;b_j) = \phi(M; b_j)$, so in particular 
    $\bigcap\limits_{k \leq n} \phi(M; b_k) = \bigcap\limits_{k \leq n, k \neq i} \phi(M;b_k)$.
    This shows that $S$ has breadth $n$.
\end{proof}

\section{Finite structures}

In this section, we study the VCd property in finite partial orders.
In particular, we give an optimal upper bound on $d$ depending only on the width of the order.

\begin{lemma}
Let $M$ a structure in some language $\cL$ and $B\subseteq M^k$ some
set. Let also $x$ a single variable, $y$ a variable of size $k$, $\psi\left(x\right)$ an $\cL$-formula with $m$ paramaters from $B$ such that $\left|\psi\left(M\right)\right|\leq2^{d+1}-1$,
$\varphi\left(x,y\right)$ some formula.

Then for any $c\in\psi\left(M\right)$, $\text{tp}_{\varphi}\left(c/B\right)$
is definable with at most $m+d$ parameters, all taken from $B$.
\end{lemma}
\begin{proof}
By induction on $d$. If $d=0$ then $\left|\psi\left(M\right)\right|\leq1$
thus if $c\in\psi\left(M\right)$ then $\psi\left(M\right)=\left\{ c\right\} $
thus $\exists x\psi\left(x\right)\wedge\varphi\left(x,y\right)$ is
a definition of $\text{tp}_{\varphi}\left(c/B\right)$ with $m$ parameters
from $B$.

Assume the claim holds for $d-1\geq0$ and $\left|\psi\left(M\right)\right|\leq2^{d+1}-1$.
We split the proof into 3 cases:
\begin{enumerate}
\item There is some $a\in\varphi\left(c,B\right)$ such that $\left|\left\{ c'\in\psi\left(M\right)\mid\varphi\left(c',a\right)\right\} \right|\leq2^{d}-1$.
Then $\psi\left(x\right)\wedge\varphi\left(x,a\right)$ is a formula
with at most $m+1$ parameters from $B$ of size $\leq2^{\left(d-1\right)+1}-1$
thus by induction $\text{tp}_{\varphi}\left(c/B\right)$ is definable
with at most $m+1+d-1=m+d$ parameters from $B$.
\item There is some $a\in B\backslash\varphi\left(c,B\right)$ such that
$\left|\left\{ c'\in\psi\left(M\right)\mid\neg\varphi\left(c',a\right)\right\} \right|\leq2^{d}-1$.
Then $\psi\left(x\right)\wedge\neg\varphi\left(x,a\right)$ is a formula
with at most $m+1$ parameters of size $\leq2^{\left(d-1\right)+1}-1$
and again we proceed by induction.
\item If neither 1 nor 2 holds, this means that for any $a\in B$, if $a\in\varphi\left(c,B\right)$
then $\left|\left\{ c'\in\psi\left(M\right)\mid\varphi\left(c',a\right)\right\} \right|\geq2^{d}$
while if $\left|\left\{ c'\in\psi\left(M\right)\mid\varphi\left(c',a\right)\right\} \right|\geq2^{d}$
then $\left|\left\{ c'\in\psi\left(M\right)\mid\neg\varphi\left(c',a\right)\right\} \right|\leq2^{d}-1$
thus $a\notin B\backslash\varphi\left(c,B\right)$ that is $a\in\varphi\left(c,B\right)$.
We conclude that 
\[
\exists x_{1},...,x_{2^{d}}:\left(\bigwedge_{i}\psi\left(x_{i}\right)\right)\wedge\left(\bigwedge_{i<j}x_{i}\neq x_{j}\right)\wedge\left(\bigwedge_{i}\varphi\left(x_{i},y\right)\right)
\]
\\
is a definition of $\text{tp}_{\varphi}\left(c/B\right)$ with at
most $m<m+d$ parameters thus we are done.
\end{enumerate}
\end{proof}

\begin{remark}
If $M$ is finite, there are only finitely many formulas up to equivalence
thus every type is definable. This means that if $\text{tp}\left(c/\emptyset\right)$
has at most $2^{d+1}-1$ realizations then for any $B$ and $\varphi$
we have $\text{tp}_{\varphi}\left(c/B\right)$ is definable with at
most $d$ parameters from $B$. 
\end{remark}

\begin{lemma}
If $M$ is a finite structure in a language $\cL$ containing a binary
relation symbol $<$ such that $\left(M,<\right)$ is a poset, then
every $\emptyset$ type is an antichain.
\end{lemma}
\begin{proof}
If $a<b$ then $\left|\left\{ x\in M\mid x<a\right\} \right|<\left|\left\{ x\in M\mid x<b\right\} \right|$,
and $\left|\left\{ x\in M\mid x<a\right\} \right|=n$ is definable.
\end{proof}

\begin{corollary}
If $M$ is a finite structure in a language $\cL$ containing a binary
relation symbol $<$ such that $\left(M,<\right)$ is a poset of width
$2^{d+1}-1$, then $M$ has $VCd$.

In other words, a finite order of width $n$ has VCd for $d=\lfloor\log_2(n)\rfloor$.
\end{corollary}

\begin{remark}
    This corollary can be considered a generalization of $(3)\Rightarrow(1)$ in \cite{aschenbrennervc}, Lemma 5.4.
\end{remark}

\begin{example}
The bound $2^{d+1}-1$ is tight. 

Consider $P=\left\{ \pm1\right\} ^{d+1}$ and let $H_{i,\varepsilon}=\left\{ a\in P\mid a_{i}=\varepsilon\right\} $.
Let $\mathcal{H}=\left\{ H_{i,\varepsilon}\mid i\leq d,\varepsilon=\pm1\right\} $
and consider $M=\left(P\cup\mathcal{H},\in\right)$ as a poset. Then
there are exactly 2 types over $\emptyset$ ---- the type of some
$a\in P$ (of size $2^{d+1}$) and the type of some $H_{i,\varepsilon}$
(of size $2\left(d+1\right)$). 

For any $i\leq d$, the mapping \[\sigma_{i}\left(x\right)=\begin{cases}
\left(x_{0},...,x_{i-1},-x_{i},x_{i+1},...,x_{d}\right) & x\in P\\
x & x=H_{j,\varepsilon},j\neq i\\
H_{i,-\varepsilon} & x=H_{i,\varepsilon}
\end{cases}\] is an automorphism of $M$ fixing no element of $P$ and all but
two $H_{j,\varepsilon}$.

Let $B=\mathcal{H}$, and take $c$ the constant sequence $1$ in
$P$. Then for any $B'\subseteq B$ of size at most $d$, there must
be some $i$ such that $\left\{ H_{i,1},H_{i,-1}\right\} \cap B'=\emptyset$
thus $\sigma_{i}$ fixes $B'$ but not $\text{tp}_{x\in y}\left(c/B\right)$,
so $\text{tp}_{x\in y}\left(c/B\right)$ cannot be definable over
$B'$.
\end{example}

\begin{remark}
    In particular, the above example shows that there is no $d < \omega$ such that any (finite)
    partial order of finite width has VCd.
\end{remark}

\bibliographystyle{alpha}
\bibliography{references}

\newcommand{\etalchar}[1]{$^{#1}$}
\begin{thebibliography}{ADH{\etalchar{+}}16}

\bibitem[ADH{\etalchar{+}}16]{aschenbrennervc}
Matthias Aschenbrenner, Alf Dolich, Deidre Haskell, Dugald Macpherson, and Sergei Starchenko.
\newblock Vapnik-chervonenkis density in some theories without the independence property, i.
\newblock {\em Transactions of the American Mathematical Society}, 368(8):5889--5949, 2016.

\end{thebibliography}

\end{document}